\documentclass{amsart}

\usepackage{latexsym}
\usepackage{amsmath}
\usepackage{epsfig}
\usepackage{amssymb}
\usepackage{enumerate}
\usepackage{color}
\usepackage{expdlist}

\usepackage{a4wide}

\newtheorem{theorem}{Theorem}
\newtheorem{lemma}{Lemma}

\newtheorem{corollary}{Corollary}

\newcommand{\cB}{{\mathcal B}}

\newcommand{\cF}{{\mathcal F}}

\newcommand{\cH}{{\mathcal H}}

\newcommand{\cL}{{\mathcal L}}
\newcommand{\cN}{{\mathcal N}}

\newcommand{\cR}{{\mathcal R}}
\newcommand{\cS}{{\mathcal S}}
\newcommand{\cT}{{\mathcal T}}
\newcommand{\cA}{{\mathcal A}}

\newcommand{\opt}{{\rm OPT(\cB_\cT\mbox{-\rm splitting})}}

\title[Approximating Hybridization number and Directed Feedback Vertex Set]{Cycle killer... qu'est-ce que c'est? On the comparative approximability of hybridization number and directed feedback vertex set}

\author{Steven Kelk, Leo van Iersel, Nela Leki\'{c}, Simone Linz, Celine Scornavacca, Leen Stougie}
\thanks{Leo van Iersel was supported by a Veni grant of The Netherlands Organisation for Scientific Research (NWO)}


\date{\today}

\begin{document}

\begin{abstract}
\noindent
We show that the problem of computing the hybridization number of two rooted binary
phylogenetic trees on the same set of taxa $X$ has a constant factor polynomial-time approximation if and only if the problem of computing a minimum-size feedback vertex set in
a directed graph (DFVS) has a constant factor polynomial-time approximation. The latter problem, which asks for a minimum number of vertices to be removed from a directed graph to transform it
into a directed acyclic graph, is one of the problems in Karp's seminal 1972 list of 21 NP-complete problems. However, despite considerable attention from the combinatorial optimization community
it remains to this day unknown whether a constant factor polynomial-time approximation exists for DFVS. Our result thus places the (in)approximability of hybridization number in a much broader complexity context, and as a consequence we obtain that hybridization number inherits inapproximability results from the problem Vertex Cover. On the positive side, we use results from the DFVS literature to give an $\text{O}( \log r \log \log r)$ approximation for hybridization number, where $r$ is the value of an optimal solution to the hybridization number problem.
\end{abstract}

\maketitle

\section{Introduction}
\noindent
The traditional model for representing the evolution of a set  of species $X$ (or, more generally, a set of \emph{taxa}) is the \emph{rooted phylogenetic tree} \cite{MathEvPhyl,reconstructingevolution,SempleSteel2003}. Essentially, this is a rooted tree where the leaves are bijectively labelled by $X$ and the edges are directed away from the unique root. A \emph{binary} rooted phylogenetic tree carries the additional restriction that
the root has indegree zero and outdegree two, leaves have indegree one and outdegree zero,
and all other (internal) vertices have indegree one and outdegree two. Rooted binary phylogenetic trees will have a central role in this article.

In recent years there has been a growing interest in extending the phylogenetic tree model to also incorporate non-treelike evolutionary phenomena such as hybridizations, recombinations and horizontal gene transfers. This has stimulated research into \emph{rooted phylogenetic networks} which generalize rooted phylogenetic trees by also permitting vertices with indegree two or higher, called \emph{reticulation} vertices, or simply \emph{reticulations}. For detailed background information on phylogenetic networks we refer the reader to \cite{husonetalgalled2009,HusonRuppScornavacca10,surveycombinatorial2011,twotrees,Nakhleh2009ProbSolv,Semple2007}.
In a rooted \emph{binary} phylogenetic network the reticulation vertices are all indegree two and outdegree one (and all other vertices obey the usual restrictions of a rooted binary phylogenetic tree).

\begin{figure}
    \centering
    \includegraphics[scale=.5]{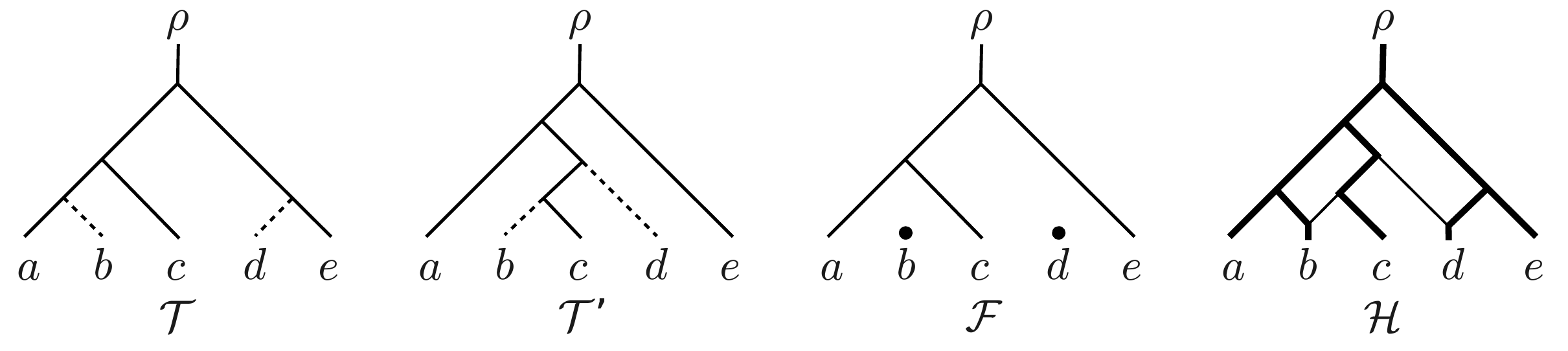}
    \caption{Two phylogenetic trees,~$\cT$ and~$\cT'$, an acyclic agreement forest~$\cF$ for~$\cT$ and~$\cT'$ and a hybridization network~$\cH$ that displays~$\cT$ and~$\cT'$, and has hybridization number~2. All edges are directed away from the root~$\rho$. Forest~$\cF$ can be obtained from either of~$\cT$ and~$\cT'$ by deleting the dashed edges. Bold edges are used in~$\cH$ to illustrate that this network displays~$\cT$.}
    \label{fig:intro}
\end{figure}

Informally, we say that a phylogenetic network $\cN$ on $X$ \emph{displays} a phylogenetic tree $\cT$ on $X$ if it is possible to delete all but one incoming edge of each reticulation vertex of $\cN$ such that, after subsequently suppressing vertices which have indegree and outdegree both equal to one, the tree $\cT$ is obtained (see Figure~\ref{fig:intro}). Following the publication of several seminal articles in 2004-5 (e.g. \cite{baroni05,BaroniEtAl2004}) there has been considerable research interest in the following biologically-inspired question. Given two rooted, binary phylogenetic trees $\cT$ and $\cT'$ on the same set of taxa $X$, what is the minimum number of reticulations required by a phylogenetic network $\cN$ on $X$ which displays both $\cT$ and $\cT'$? This value is often called the \emph{hybridization number} in the literature, and when addressing this specific problem the term \emph{hybridization network} is often used instead of the more general term phylogenetic network. For the purpose of consistency we will henceforth use the term hybridization network in this article.

{\sc MinimumHybridization}, the problem of computing the hybridization number, has been shown to be both NP-hard and APX-hard \cite{bordewich07a}, from which several related phylogenetic network construction techniques also inherit hardness \cite{twotrees,elusiveness}. APX-hardness means that there exists a constant $c>1$ such that the existence of a polynomial-time approximation algorithm that achieves an approximation ratio strictly smaller than $c$ would imply P=NP. As is often the case with APX-hardness results, the value $c$ given in \cite{bordewich07a} is very small, $\frac{2113}{2112}$. It is not known whether {\sc MinimumHybridization} is actually in APX, the class of problems for which polynomial-time approximation algorithms exist that achieve a constant approximation ratio. In fact, there are to date no non-trivial polynomial-time
approximation algorithms, constant factor or otherwise, for {\sc Hybridization number}. This omission stands in stark contrast to other positive results, which we now discuss briefly.

On the fixed parameter tractability (FPT) front - we refer to \cite{downey1999,Flum2006,Gramm2008,niedermeier2006} for an introduction~- a variety of increasingly sophisticated algorithms have been developed. These show that for many practical instances of {\sc MinimumHybridization} the problem can be efficiently solved (to the extent that even enumeration of \emph{all} optimum solutions is often, in practice, tractable)
\cite{bordewich2,bordewich07b,hybridnet,quantifyingreticulation,firststeps,whiddenFixed,whiddenWABI}. Secondly, the problem of computing the \emph{rooted subtree prune and regraft} (rSPR) distance, which bears at least a superficial similarity to the computation of hybridization number, permits a polynomial-time 3-approximation algorithm \cite{rsprFPT,3approxRSPR,whiddenFixed} and efficient FPT algorithms \cite{rsprFPT,whiddenFixed,whiddenRSPRexp}. Why then is it so difficult to give formal performace guarantees for approximating {\sc MinimumHybridization}?

A clue lies in the nature of the abstraction that (with very few exceptions) is used to compute hybridization numbers, the \emph{Maximum Acyclic Agreement Forest} ({\sc MAAF}), introduced in \cite{baroni05} (see Figure~\ref{fig:intro}). Roughly speaking, computing the hybridization number of two trees $\cT$ and $\cT'$ is essentially identical to the problem of cutting $\cT$ and $\cT'$ into as few vertex-disjoint subtrees as possible
such that (i) the subtrees of $\cT$ are isomorphic to the subtrees of $\cT'$ and -
critically - (ii) a specific ``reachability'' relation on these subtrees is acyclic. Condition (ii) seems to be the core of the issue, because without this condition the problem would be no different to the problem of computing the rSPR distance, which as previously mentioned seems to be comparatively
tractable. (Note that the hybridization number of two trees can in general be much larger than their rSPR distance). The various FPT algorithms for computing hybridization number deal with the unwanted cycles in the reachability relation
in a variety of ways but all resort to some kind of brute force analysis
to optimally avoid (e.g. \cite{firststeps}) or break (e.g. \cite{hybridnet,whiddenFixed}) them.

In this article we demonstrate why it is so difficult to deal with the cycles. It turns out that
{\sc MinimumHybridization} is, in an approximability sense, a close relative of the
problem Feedback Vertex Set on directed graphs (DFVS). In this problem we wish to remove a minimum number of vertices from a directed graph to transform it into a directed acyclic graph. DFVS is one of the original
NP-complete problems (it is in Karp's famous 1972 list of 21 NP-complete problems \cite{karp1972}) and is also known to be APX-hard \cite{Kann}. However, despite almost forty years of attention it is still unknown whether
DFVS permits a constant approximation ratio i.e. whether it is in APX. (The undirected
variant of FVS, in contrast, appears to be significantly more tractable. It is 2-approximable even in the weighted case \cite{Bafna1999}).

By coupling the approximability of {\sc MinimumHybridization} to DFVS we show that {\sc MinimumHybridization}
is just as hard as a problem that has so far eluded the entire combinatorial optimization community. Specifically, we show that for every constant $c > 1$ and every $\epsilon > 0$ the existence of a polynomial-time $c$-approximation for {\sc MinimumHybridization}
would imply a polynomial-time $(c + \epsilon)$-approximation for DFVS. In the other
direction we show that, for every $c > 1$, the existence of a polynomial-time $c$-approximation
for DFVS would imply a polynomial-time $6c$-approximation for {\sc MinimumHybridization}. In other words:
DFVS is in APX if and only if {\sc MinimumHybridization} is in APX. Hence a constant factor approximation algorithm for either
algorithm would be a major breakthrough in theoretical computer science.

There are several interesting spin-off consequences of this result, both negative and positive. On the negative side, it is known that there is a very simple parsimonious reduction from the classical problem Vertex Cover to DFVS \cite{karp1972}. Consequently, a $c$-approximation for DFVS entails a $c$-approximation for Vertex Cover, for every $c \geq 1$. For $c < 10\sqrt{5}-21 \approx 1.3606$ there cannot exist a polynomial-time $c$-approximation of Vertex Cover, assuming P $\neq$ NP \cite{dinur,DinurAnnals}. Also, if the Unique Games Conjecture is true then for $c < 2$ there cannot exist a polynomial-time $c$-approximation of Vertex Cover \cite{Khot2008}. (Whether Vertex Cover permits a constant factor approximation ratio strictly smaller than 2 is a long-standing open problem). The main result in this article hence not only shows that {\sc MinimumHybridization} is in APX if and only if DFVS is in APX, but also that {\sc MinimumHybridization} cannot be approximated within a factor of 1.3606, unless P=NP (and not within a factor smaller than 2 if the Unique Games Conjecture is true). This improves significantly on the current APX-hardness threshhold of $\frac{2113}{2112}$.

On the positive side, we observe that already-existing approximation algorithms for {\sc DFVS} can be utilised to give
asymptotically comparable approximation ratios for {\sc MinimumHybridization}. To date the best polynomial-time approximation algorithms for {\sc DFVS} achieve an approximation ratio of $\text{O}( \min\{ \log n \log \log n, \log \tau^{*} \log \log \tau^{*}\} )$, where $n$ is the number of vertices in the graph and $\tau^{*}$ is the optimal fractional solution of the problem (taking the weights of the vertices into account) \cite{dfvsApprox,dfvsSeymour}. We show that this algorithm can be used to give an $\text{O}( \log r \log \log r )$-approximation algorithm for {\sc MinimumHybridization}, where $r$ is the hybridization number of the two input trees. To the best of our knowledge, this is the first non-trivial polynomial-time approximation algorithm for {\sc MinimumHybridization}.

The main result also has interesting consequences for the fixed parameter tractability of {\sc MinimumHybridization}. The inflation factor of 6 in the reduction from {\sc DFVS} to {\sc MinimumHybridization} is very closely linked to a reduction described in \cite{bordewich07b}. The authors in that article showed that the input trees can be reduced to produce a weighted instance containing at most $14r$ taxa. (The fact that the reduced instance is weighted means it cannot be automatically used to obtain a constant-factor approximation algorithm). In this article we sharpen their analysis to show that the reduction they describe actually produces a weighted instance with at most $9r$ taxa. Without this sharpening, the inflation factor we obtain would have been higher than 6. From this analysis it becomes clear that the kernel size has an important role to play in analysing the approximability of {\sc MinimumHybridization}.

This raises some interesting general questions about the linkages between {\sc MinimumHybridization} and {\sc DFVS}. For example, it can be shown that, in a formal sense, a small modification to the reduction described in \cite{bordewich07b} produces a kernel (without weights) of quadratic size. This contrasts sharply with {\sc DFVS}. It is known that {\sc DFVS} is fixed parameter tractable~\cite{dfvsFPT}, but it is \emph{not} known whether {\sc DFVS} permits a polynomial-size kernel. Might {\sc MinimumHybridization} give us new insights into the structure of {\sc DFVS} (and vice-versa)? More generally: within which complexity frameworks is one of the two problems strictly harder than the other?

The structure of this article is as follows. In the next section, we define the considered problems formally and describe the reductions that were used to show that {\sc MinimumHybridization} is fixed parameter tractable. In Section~\ref{sec:kernel}, we show an improved bound on the sizes of reduced instances. Subsequently, we use these results to show an approximation-preserving reduction from {\sc MinimumHybridization} to {\sc DFVS} in Section~\ref{sec:2dfvs} and an approximation-preserving reduction from {\sc DFVS} to {\sc MinimumHybridization} in Section~\ref{sec:2minhybrid}.

\section{Preliminaries}\label{sec:prelim}

{\bf Phylogenetic Trees.}
Throughout the paper, let~$X$ be a finite set of \emph{taxa} (taxonomic units). A {\it rooted binary phylogenetic $X$-tree $\cT$} is a rooted tree whose root has degree two, whose interior vertices have degree three and whose leaves are bijectively labelled by the elements of~$X$. The edges of the tree can be seen as being directed away from the root. The set of leaves of~$\cT$ is denoted as~$\cL(\cT)$. We identify each leaf with its label. We sometimes call a rooted binary phylogenetic $X$-tree a \emph{tree} for short.

In the course of this paper, different types of subtrees play an important role. Let~$\cT$ be a rooted phylogenetic $X$-tree and~$X'$ a subset of~$X$. The minimal rooted subtree of~$\cT$ that connects all leaves in~$X'$ is denoted by $\cT(X')$. Furthermore, the tree obtained from $\cT(X')$ by suppressing all non-root degree-$2$ vertices is the {\it restriction of $\cT$ to $X'$} and is denoted by $\cT|X'$. Lastly, a subtree of~$\cT$ is {\it pendant} if it can be detached from~$\cT$ by deleting a single edge.

{\bf Hybridization Networks.}
A {\it hybridization network $\cH$} on a set~$X$ is a rooted acyclic directed graph, which has a single root of outdegree at least 2, has no vertices with indegree and outdegree both~1, and in which the vertices of outdegree~0 are bijectively labelled by the elements of~$X$. A hybridization network is \emph{binary} if all vertices have indegree and outdegree at most~2 and every vertex with indegree~2 has outdegree~1.

For each vertex~$v$ of $\cH$, we denote by $d^-(v)$ and $d^+(v)$ its indegree and outdegree respectively. If $(u,v)$ is an arc of $\cH$, we say that~$u$ is a \emph{parent} of~$v$ and that~$v$ is a \emph{child} of~$u$. Furthermore, if there is a directed path from a vertex~$u$ to a vertex~$v$, we say that~$u$ is an \emph{ancestor} of~$v$ and that~$v$ is a \emph{descendant} of~$u$.

A vertex of indegree greater than one represents an evolutionary event in which lineages combined, such as a hybridization, recombination or horizontal gene transfer event. We call these vertices {\it hybridization vertices}. To quantify the number of hybridization events, the {\it hybridization number} of a hybridization network $\cH$ with root $\rho$ is given by
$$h(\cH)=\sum_{v\ne\rho}(d^-(v)-1).$$
Observe that $h(\cH)=0$ if and only if~$\cH$ is a tree.

Let $\cH$ be a hybridization network on~$X$ and~$\cT$ a rooted binary phylogenetic $X'$-tree with $X'\subseteq X$. We say that~$\cT$ is {\it displayed} by~$\cH$ if~$\cT$ can be obtained from~$\cH$ by deleting vertices and edges and suppressing vertices with $d^+(v)=d^-(v)=1$ (or, in other words, if a subdivision of~$\cT$ is a subgraph of~$\cH$). Intuitively, if~$\cH$ displays~$\cT$, then all of the ancestral relationships visualized by~$\cT$ are visualized by~$\cH$.

The problem {\sc MinimumHybridization} is to compute the \emph{hybridization number} of two rooted binary phylogenetic $X$-trees~$\cT$ and~$\cT'$, which is defined as
$$h(\cT,\cT')=\min\{h(\cH): \cH \mbox{ is a hybridization network that displays } \cT \mbox{ and }\cT'\},$$
i.e., the minimum number of hybridization events necessary to display two rooted binary phylogenetic trees.

This problem can be formulated as an optimization problem as follows.

\noindent{\bf Problem:} {\sc MinimumHybridization}\\
\noindent {\bf Instance:} Two rooted binary phylogenetic $X$-trees $\cT$ and $\cT'$. \\
\noindent {\bf Solution:} A hybridization network~$\cH$ that displays~$\cT$ and~$\cT'$.\\
\noindent{\bf Objective:} Minimize $h(\cH)$.

If~$\cH$ is a hybridization network that displays~$\cT$ and~$\cT'$, then there also exists a binary hybridization network~$\cH'$ that displays~$\cT$ and~$\cT'$ such that $h(\cH)=h(\cH')$ \cite[Lemma~3]{twotrees}. Hence, we restrict our analysis to binary hybridization networks and will not emphasize again that we only deal with this kind of network.

{\bf Agreement Forests.}
A useful characterization of {\sc MinimumHybridization} in terms of agreement forests was discovered by Baroni et al.~\cite{baroni05}, building on an idea in~\cite{hein96}. Bordewich and Semple used this characterization to show that {\sc MinimumHybridization} is NP-hard. Such agreement forests play a fundamental role in this paper. For the purpose of the upcoming definition and, in fact, much of the paper, we regard the root of a tree~$\cT$ (or network~$\cH$) as a vertex~$\rho$ at the end of a pendant edge adjoined to the original root. Furthermore, we view~$\rho$ as an element of the label set of~$\cT$; thus $\cL(\cT)=X\cup\{\rho\}$.

Let $\cT$ and $\cT'$ be two rooted binary phylogenetic $X$-trees. {A partition $\cF=\{\cL_\rho,\cL_1,\cL_2,\ldots,\cL_k\}$ of $X\cup\{\rho\}$ is an {\it agreement forest} for~$\cT$ and~$\cT'$ if $\rho\in\cL_\rho$ and the following conditions are satisfied:}
\begin{itemize}
\item[(1)] for all $i\in\{\rho,1,2,\ldots,k\}$, we have $\cT|\cL_i\cong\cT'|\cL_i$, and
\item[(2)] the trees in $\{\cT(\cL_i): i\in\{\rho,1,2,\ldots,k\}\}$ and $\{\cT'(\cL_i): i\in\{\rho,1,2,\ldots,k\}\}$ are vertex-disjoint subtrees of $\cT$ and $\cT'$, respectively.
\end{itemize}
In the definition above, the notation $\cong$ is used to denote a graph isomorphism that preserves leaf-labels.

Note that, even though an agreement forest is formally defined as a partition of the leaves, we often see the collection of trees $\{\cT|\cL_\rho,\cT|\cL_1,\ldots,\cT|\cL_k\}$ as the agreement forest. So, intuitively, an agreement forest for~$\cT$ and~$\cT'$ can be seen as a collection of trees that can be obtained from either of~$\cT$ and~$\cT'$ by deleting a set of edges and subsequently ``cleaning up'' by deleting unlabelled vertices and suppressing indegree-1 outdegree-1 vertices (see Figure~\ref{fig:chains}). Therefore, we often refer to the elements of an agreement forest as \emph{components}.

The \emph{size} of an agreement forest~$\cF$ is defined as its number of elements (components) and is denoted by~$|\cF|$.

A characterization of the hybridization number $h(\cT,\cT')$ in terms of agreement forests requires an additional condition. Let $\cF=\{\cL_\rho,\cL_1,\cL_2,\ldots,\cL_k\}$ be an agreement forest for $\cT$ and $\cT'$. Let $G_{\cF}$ be the directed graph that has vertex set $\cF$ and an edge $(\cL_i,\cL_j)$ if and only if $i\neq j$ and at least one of the two following conditions holds 
\begin{itemize}
\item[(1)] the root of $\cT(\cL_i)$ is an ancestor of the root of $\cT(\cL_j)$ in $\cT$;
\item[(2)] the root of $\cT'(\cL_i)$ is an ancestor of the root of $\cT'(\cL_j)$ in $\cT'$.
\end{itemize}
The graph $G_{\cF}$ is called the {\it inheritance graph} associated with $\cF$. We call $\cF$ an {\it acyclic agreement forest} for $\cT$ and $\cT'$ if $G_{\cF}$ has no directed cycles. If~$\cF$ contains the smallest number of elements (components) over all acyclic agreement forests for $\cT$ and $\cT'$, we say that $\cF$ is a {\it maximum acyclic agreement forest} for $\cT$ and $\cT'$. Note that such a forest is called a \emph{maximum} acyclic agreement forest, even though one \emph{minimizes} the number of elements, because in some sense the ``agreement'' is maximized. (Also note that acyclic agreement forests were called \emph{good} agreement forests in~\cite{baroni05}.)

We define $m_a(\cT, \cT')$ to be the number of elements of a maximum acyclic agreement forest for~$\cT$ and~$\cT'$ minus one. Also the problem of computing $m_a(\cT, \cT')$ has an optimization counterpart:

\noindent{\bf Problem:} Maximum Acyclic Agreement Forest ({\sc MAAF})\\
\noindent {\bf Instance:} Two rooted binary phylogenetic $X$-trees $\cT$ and $\cT'$. \\
\noindent {\bf Solution:} An acyclic agreement forest $\cF$ for $\cT$ and $\cT'$. \\
\noindent {\bf Objective:} Minimize $|\cF|-1$.

We minimize $|\cF|-1$, rather than~$|\cF|$, following~\cite{bordewich07a}, because $|\cF|-1$ corresponds to the number of edges one needs to remove from either of the input trees to obtain~$\cF$ (after ``cleaning up'') and because of the relation we describe below between this problem and {\sc MinimumHybridization}. Nevertheless, it can be shown that, from an approximation perspective, it does not matter whether one minimizes~$|\cF|$ or~$|\cF|-1$ (which is not obvious).

\begin{theorem}{\cite[Theorem 2]{baroni05}}\label{t:hybrid}
Let $\cT$ and $\cT'$ be two rooted binary phylogenetic $X$-trees. Then
$$h(\cT,\cT')=m_a(\cT,\cT').$$
\end{theorem}

\noindent It is this characterization that was used by Bordewich and Semple~\cite{bordewich07a} to show that {\sc MinimumHybridization} is NP-hard. To show that also an approximation for one problem can be used to approximate the other problem, one needs the following slightly stronger result.

\begin{theorem}\label{t:hybrid2} Let $\cT$ and $\cT'$ be two rooted binary phylogenetic $X$-trees. Then
\begin{enumerate}
\item[(i)] from a hybridization network~$\cH$ that displays~$\cT$ and~$\cT'$, one can construct in polynomial time an acyclic agreement forest~$\cF$ for~$\cT$ and~$\cT'$ such that $|\cF|-1\leq h(\cH)$ and
\item[(ii)] from an acyclic agreement forest~$\cF$ for~$\cT$ and~$\cT'$, one can construct in polynomial time a hybridization network~$\cH$ that displays~$\cT$ and~$\cT'$ such that  $h(\cH)\leq |\cF|-1$.
\end{enumerate}
\end{theorem}

This result follows from the proof of \cite[Theorem~2]{baroni05} using the observation above that we may assume that~$\cH$ is binary.

We now formally introduce the last optimization problem discussed in this paper. A \emph{feedback vertex set (FVS)} of a directed graph~$D$ is a subset of the vertices that contains at least one vertex from each directed cycle in~$D$. Equivalently, a subset~$V'$ of the vertices of~$D$ is a feedback vertex set if and only if removing~$V'$ from~$D$ gives a directed acyclic graph. The \emph{minimum feedback vertex set problem on directed graphs} (DFVS) is defined as: given a directed graph~$D$, find a feedback vertex set of~$D$ that has minimum size.

{\bf Reductions and Fixed Parameter Tractability.}
After establishing the NP-hardness of {\sc MinimumHybridization}, the same authors showed that this problem is also fixed parameter tractable~\cite{bordewich07b}. They show how to reduce a pair of rooted binary phylogenetic $X$-trees~$\cT$ and~$\cT'$, such that the number of leaves of the reduced trees is bounded by $14h(\cT,\cT')$, whence a brute-force algorithm can be used to solve the reduced instance, giving a fixed parameter tractable algorithm.

To describe the reductions, we need some additional definitions. Let $\cT$ be a rooted binary phylogenetic $X$-tree. For $n\geq 2$, an {\it $n$-chain} of $\cT$ is an $n$-tuple $(a_1,a_2,\ldots,a_n)$ of elements of $\cL(\cT)\setminus\{\rho\}$ such that the parent of $a_1$ is either the same as the parent of $a_2$ or the parent of $a_1$ is a child of the parent of $a_2$ and, for each $i\in\{2,3,\ldots,n-1\}$, the parent of $a_i$ is a child of the parent of $a_{i+1}$; i.e., the subgraph induced by $a_1,a_2,\ldots,a_n$ and their parents is a {\em caterpillar} (see Figure~\ref{fig:chains}).

Now, let $A=(a_1,a_2,\ldots,a_n)$ be an $n$-chain that is common to two rooted binary phylogenetic $X$-trees $\cT$ and $\cT'$ {with $n\geq 2$}, and let $\cF$ be an acyclic agreement forest for $\cT$ and $\cT'$. We say that $A$ {\it survives} in $\cF$ if there exists an element in $\cF$ that is a superset of $\{a_1,a_2,\ldots,a_n\}$, while we say that $A$ is {\it atomized} in $\cF$ if each element in $\{a_1,a_2,\ldots,a_n\}$ is a singleton in $\cF$ (see Figure~\ref{fig:chains}). Furthermore, if~$T$ is a common pendant subtree of $\cT$ and $\cT'$, then we say that~$T$ \emph{survives} in~$\cF$ if there is an element of~$\cF$ that is a superset of the label set of~$T$.

\begin{figure}
    \centering
    \includegraphics[scale=.5]{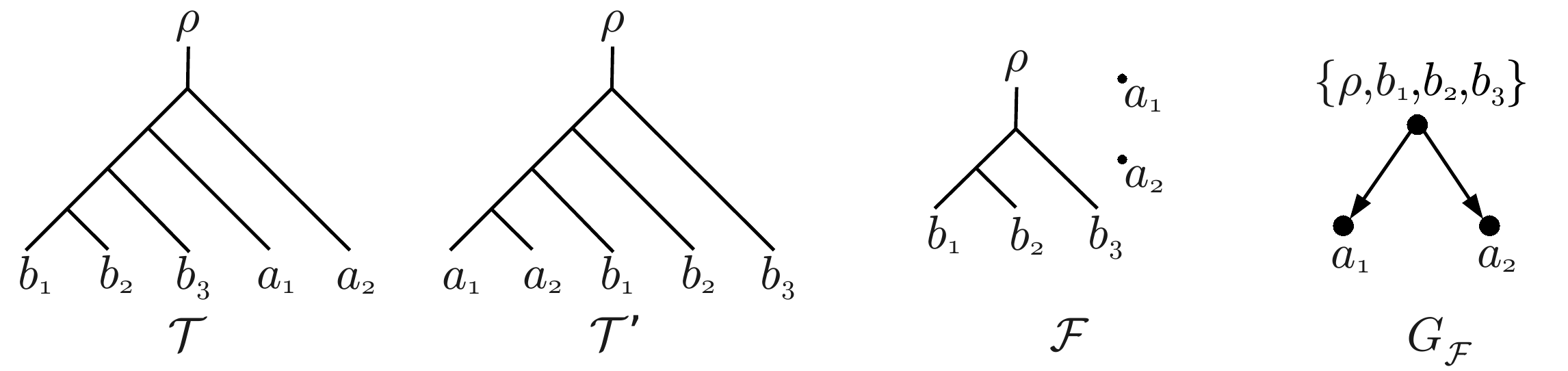}
    \caption{Two input trees~$\mathcal{T}$ and $\mathcal{T}'$, an agreement forest~$\mathcal{F}$ for~$\mathcal{T}$ and $\mathcal{T}'$ and the inheritance graph~$G_\mathcal{F}$. The trees have two common chains: $(a_1,a_2)$ and $(b_1,b_2,b_3)$. In the agreement forest~$\mathcal{F}$, chain $(a_1,a_2)$ is atomized while chain $(b_1,b_2,b_3)$ survives. The agreement forest~$\mathcal{F}$ is acyclic because~$G_\mathcal{F}$ is acyclic.}
    \label{fig:chains}
\end{figure}

The following lemma basically shows that we can reduce subtrees and chains. It differs slightly from the corresponding lemma in~\cite{bordewich07b} because we consider approximations while Bordewich and Semple considered only optimal solutions in that paper.

\begin{lemma}\label{lem:survives}
Let~$\cF$ be an acyclic agreement forest for two trees~$\cT$ and~$\cT'$. Then there exists an acyclic agreement forest~$\cF'$ for~$\cT$ and~$\cT'$ with~$|\cF'|\leq |\cF|$ such that
\begin{itemize}
\item[1.] every common pendant subtree of~$\cT$ and~$\cT'$ survives in~$\cF'$ and
\item[2.] every common $n$-chain of~$\cT$ and~$\cT'$, with~$n\geq 3$, either survives or is atomized in~$\cF'$.
\end{itemize}
Moreover, $\cF'$ can be obtained from~$\cF$ in polynomial time.
\end{lemma}
\begin{proof}
Follows from the proof of~\cite[Lemma 3.1]{bordewich07b}. There are two differences with~\cite[Lemma 3.1]{bordewich07b}. Firstly, our result is slightly simpler because we consider two unweighted trees~$\cT$ and~$\cT'$, while the authors of~\cite{bordewich07b} allow the unreduced trees~$\cT$ and $\cT'$ to already have weights on 2-chains. Secondly,~\cite[Lemma 3.1]{bordewich07b} only shows the result for optimal agreement forests. However, a careful analysis of the proof of~\cite[Lemma 3.1]{bordewich07b} shows that it can also be used to prove this lemma.
\end{proof}

We are now ready to formally describe the aforementioned tree reductions. Let $\cT$ and $\cT'$ be two rooted binary phylogenetic $X$-trees, $P$ a set that is initially empty and $w:P\rightarrow\mathbb{Z}^+$ a weight function on the elements in $P$.

\noindent{\bf Subtree Reduction.} Replace any maximal pendant subtree with at least two leaves that is common to $\cT$ and $\cT'$ by a single leaf with a new label.

\noindent{\bf Chain Reduction.} Replace any maximal $n$-chain $(a_1,a_2,\ldots,a_n)$, with~$n\geq 3$, that is common to $\cT$ and $\cT'$ by a $2$-chain with new labels $a$ and $b$. Moreover, add a new element $(a,b)$ with weight $w(a,b)=n-2$ to $P$.

Let $\cS$ and $\cS'$ be two rooted binary phylogenetic $X'$-trees that have been obtained from $\cT$ and $\cT'$ by first applying subtree reductions as often as possible and then applying chain reductions as often as possible. We call $\cS$ and $\cS'$ the {\it reduced tree pair} with respect to  $\cT$ and $\cT'$. Note that a reduced tree pair always has an associated set $P$ that contains one element for each chain reduction applied. Note that $\cS$ and $\cS'$ are unambiguously defined (up to the choice of the new labels) because maximal common pendant subtrees do not overlap and maximal common chains do not overlap. Moreover, applications of the chain reduction can not create any new common pendant subtrees with at least two leaves. Hence, it is not necessary to apply subtree reductions again after the chain reductions.

Recall that every common $n$-chain, with~$n\geq 3$, either survives or is atomized (Lemma~\ref{lem:survives}). In~$\cS$ and~$\cS'$, such chains have been replaced by weighted 2-chains. Therefore, we are only interested in acyclic agreement forests for~$\cS$ and~$\cS'$ in which these weighted 2-chains either survive or are atomized. We therefore introduce a third notion of an agreement forest. Recall that~$P$ is the set of reduced (i.e. weighted) 2-chains. We say that an agreement forest~$\cF$ for~$\cS$ and~$\cS'$ is {\it legitimate} if it is acyclic and every chain~$(a,b)\in P$ either survives or is atomized in~$\cF$.

\noindent Let $\cF$ be an agreement forest for $\cS$ and $\cS'$. The {\it weight} of $\cF$, denoted by $w(\cF)$, is defined to be
\[w(\cF)=|\cF|-1+\sum_{(a,b)\in P: \mbox{ }(a,b)\textnormal{ is atomized in }\cF} w(a,b). \]
Lastly, we define $f(\cS,\cS')$ to be the minimum weight of a legitimate agreement forest for $\cS$ and $\cS'$.

Then, the following lemma says that computing the hybridization number of~$\cT$ and~$\cT'$ is equivalent to computing the minimum weight of a legitimate agreement forest for $\cS$ and $\cS'$. The second part of the lemma is necessary to show that an approximation to a reduced instance $\cS$ and $\cS'$ can be used to obtain an approximation to the original instance $\cT$ and $\cT'$.

\begin{lemma}\label{l:preserving}
Let $\cT$ and $\cT'$ be a pair of rooted binary phylogenetic $X$-trees and let $\cS$ and $\cS'$ be the reduced tree pair with respect to $\cT$ and $\cT'$. Then
\begin{enumerate}
\item[(i)] $h(\cS,\cS') \leq f(\cS,\cS') = h(\cT,\cT')$ and
\item[(ii)] given a legitimate agreement forest~$\cF_S$ for~$\cS$ and~$\cS'$, we can find, in polynomial time, an acyclic agreement forest~$\cF$ for $\cT$ and $\cT'$ such that $|\cF| - 1 = w(\cF_S)$.
\end{enumerate}
\end{lemma}
\begin{proof}
In part (i), the inequality follows directly from the definition of~$f$ while the equality is equivalent to \cite[Proposition 3.2]{bordewich07b} if the unreduced trees~$\cT$ and $\cT'$ are unweighted (i.e. if~$P$ is initially empty). Part (ii) follows from the proof of \cite[Proposition 3.2]{bordewich07b}.
\end{proof}

The fixed parameter tractability of {\sc MinimumHybridization} now follows from the next lemma, which bounds the number of leaves in a reduced tree pair.

\begin{lemma}\cite[Lemma 3.3]{bordewich07b}\label{l:oldFPTfactor}
Let $\cT$ and $\cT'$ be two rooted binary phylogenetic $X$-trees, $\cS$ and $\cS'$ the reduced tree pair with respect to $\cT$ and $\cT'$, and~$X'$ the label set of~$\cS$ and~$\cS'$. If $h(\cT,\cT')>0$, then $|X'|<14h(\cT,\cT')$.
\end{lemma}

\noindent We show in Section~\ref{sec:kernel} that the reduced trees have at most $9h(\cT,\cT')$ leaves. This improved
bound will be important in the approximation-preserving reductions we give later in the paper.

\section{An improved bound on the size of reduced instances of {\sc MinimumHybridization}}\label{sec:kernel}

We start with some definitions and an intermediate result. The bound on the size of the reduced instance will be proven in Theorem \ref{thm:newKernel}.

An $r$-\emph{reticulation generator} (for short, {\it $r$-generator}) is defined to be a directed acyclic multigraph with  a single vertex of indegree~0 and  outdegree~1, precisely~$r$ reticulation vertices (indegree~2 and outdegree at most~1), and apart from that only vertices of indegree~1 and outdegree~2 \cite{KelkScornavacca2011}.  The \emph{sides} of an $r$-generator are defined as the union of its edges (the \emph{edge} sides) and its vertices of indegree-2 and outdegree-0 (the \emph{node sides}). Adding a set of labels $L$ to an edge side $(u, v)$ of an $r$-generator involves subdividing $(u, v)$ to a path of $|L|$ internal vertices and, for each such internal vertex $w$, adding a new leaf $w'$, an edge $(w, w')$, and labeling $w'$ with some taxon from $L$ (such that~$L$ bijectively labels the new leaves). On the other hand, adding a label $l$ to a node side $v$ consists of adding a new leaf $y$, an edge $(v, y)$ and labeling $y$ with $l$.

\begin{lemma}\label{lem:generator}
Let $\cT$ and $\cT'$ be two rooted binary phylogenetic $X$-trees with no common pendant subtrees with at least~2 leaves and let~$\cH$ be a hybridization network that displays~$\cT$ and~$\cT'$ with a minimum number of hybridization vertices. Then the network~$\cH'$ obtained from~$\cH$ by deleting all leaves and suppressing each resulting vertex~$v$ with $d^+(v)=d^-(v)=1$ is an $h(\cH)$-generator.
\end{lemma} 
\begin{proof}
By construction, $\cH'$ contains the same number of hybridization vertices as~$\cH$. Additionally, 
by the definition of a binary hybridization network, no vertex has indegree~2 and outdegree greater than~1, indegree greater than~2, or indegree and outdegree both~1. Now, we claim that~$\cH'$ does not have any vertex with indegree~1 and outdegree~0. To see that this holds, suppose that there exists a vertex~$v$ in~$\cH'$ such that $d^-(v)= 1$ and $d^+(v)= 0$. Then~$v$ has two children in~$\cH$. Since $d^+(v)= 0$ in~$\cH'$, no hybridization vertex can be reached by a directed path from~$v$ in~$\cH$. This means that the subnetwork of~$\cH$ rooted at~$v$ is actually a rooted tree, contradicting the fact that~$\cT$ and~$\cT'$ do not have any common pendant subtree with two or more leaves. We may thus conclude that~$\cH'$ conforms to the definition of an $h(\cH)$-generator.
\end{proof}

Reversely, by inverting the operations of suppression and deletion,~$\cH$ can be obtained from the $h(\cH)$-generator~$\cH'$ associated with~$\cH$ by adding leaves to its sides (in the sense described at the start of this section).\footnote{A similar technique was described in~\cite{KelkScornavacca2011} in a somewhat different context.}\\

\begin{theorem}\label{thm:newKernel}
Let $\cT$ and $\cT'$ be two rooted binary phylogenetic $X$-trees and~$\cS$ and~$\cS'$ the reduced tree pair on~$X'$ with respect to~$\cT$ and~$\cT'$. If $h(\cT,\cT')>0$, then $|X'|< 9h(\cT,\cT')$.
\end{theorem}
\begin{proof}
Let $\cH'$ be the $h(\cH)$-generator that is associated with a hybridization network $\cH$ for $\cS$ and $\cS'$ whose number of hybridization vertices is minimized, i.e., $h(\cH)=h(\cS,\cS')$. By definition, $\cH'$ has the following vertices:
\begin{itemize}
\item $r=h(\cH)$ reticulations; in particular $r_0$ reticulations with indegree 2 and outdegree 0 and $r_1$ reticulations with indegree 2 and outdegree 1,
\item $s$ vertices with indegree 1 and outdegree 2, and
\item one root vertex with indegree 0 and outdegree 1.
\end{itemize}
The total indegree of~$\cH'$ is $2r_0 + 2r_1 + s$. The total outdegree of~$\cH'$ is $ r_1 + 2s + 1$. Hence, $2r_0 + 2r_1 + s = r_1 + 2s + 1$ implying $s = 2r_0 + r_1 - 1$. Moreover, the total number of edges of~$\cH'$, $|E(\cH')|$, equals the total indegree and, therefore,
\begin{equation}\label{eq:edges}
|E(\cH')|=2r_0 + 2r_1 + s=2r_0 + 2r_1 +2r_0 + r_1 - 1=4r_0+3r_1-1.\\
\end{equation}

Note that for each of the $r_0$ node sides~$v$ in $\cH'$ the child of $v$ in $\cH$ is a single leaf. Moreover, each edge side in $\cH'$ cannot correspond to a directed path in $\cH$ that consists of more than three edges since, otherwise, $\cS$ and $\cS'$ would have a common $n$-chain, with $n\geq 3$. Thus,~$\cH$ can have at most two leaves per edge side of~$\cH'$ and one leaf per node side of~$\cH'$. Thus, the total number of leaves~$|X'|$ of~$\cH$ is bounded by

\begin{align*}
|X'|&\leq 2|E(\cH')| + r_0 \\
&= 2(4r_0 + 3r_1 - 1) + r_0\\
&= 9r_0 + 6r_1 - 2\\
& \leq 9r - 2\\
&< 9h(\cS,\cS')\\
&\leq 9h(\cT,\cT'),
\end{align*}
where the last inequality follows from Lemma~\ref{l:preserving}.
\end{proof}

\section{An approximation-preserving reduction from {\sc MinimumHybridization} to {\sc DFVS}}\label{sec:2dfvs}

We start by proving the following theorem, which refers to {\sc wDFVS}, the \emph{weighted} variant of {\sc DFVS} where every vertex is attributed a weight and the weight of a feedback vertex set is simply the sum of the weights of its constituent vertices. Later in the section we will prove a corresponding result for {\sc DFVS}.

\begin{theorem}\label{t:Hybrid2DFVS} If, for some~$c\geq 1$, there exists a polynomial-time $c$-approximation for {\sc wDFVS}, then there exists a polynomial-time $6c$-approximation for {\sc MinimumHybridization}.
\end{theorem}

Throughout this section, let~$\cT$ and~$\cT'$ be two rooted binary phylogenetic $X$-trees, and let $\cS$ and $\cS'$ be the reduced tree pair on $X'$ with respect to $\cT$ and $\cT'$. Using Lemma~\ref{lem:survives}, we assume throughout this section without loss of generality that~$\cT$ and~$\cT'$ do not contain any common pendant subtrees with at least two leaves. Thus, the reduced tree pair $\cS$ and $\cS'$ can be obtained from $\cT$ and $\cT'$ by applying the chain reduction only.

Before starting the proof, we need some additional definitions and lemmas. We say that a common chain $(a,b)$ of $\cS$ and $\cS'$ is a {\it reduced chain} if it is not a common chain of $\cT$ an $\cT'$. Otherwise, $(a,b)$ is an {\it unreduced chain}. Furthermore, a taxon $\ell\in X'\cup\{\rho\}$, is a {\it non-chain taxon} if it does not label a leaf of a reduced or unreduced chain of $\cS$ and $\cS'$. Now, let $\cB_\cS$ be the forest that exactly contains the following elements:
\begin{enumerate}
\item for each non-chain taxon $\ell$ of $\cS$ and $\cS'$, a {\it non-chain element} $\{\ell\}$, and
\item for each reduced and unreduced chain $(a,b)$ of $\cS$ and $\cS'$, an element $\{a,b\}$.
\end{enumerate}
Clearly, $\cB_\cS$ is an agreement forest for $\cS$ and $\cS'$, and we refer to it as a {\it chain forest} for $\cS$ and $\cS'$.
Now, obtain $\cB_\cT$ from $\cB_\cS$ by replacing  each element in $\cB_\cS$ that contains two labels of a reduced chain, say $(a,b)$, of $\cS$ and $\cS'$ with the label set that precisely contains all labels of the common $n$-chain that has been reduced to $(a,b)$ in the course of obtaining $\cS$ and $\cS'$ from $\cT$ and $\cT'$, respectively. The set $\cB_\cT$ is an agreement forest for $\cT$ and $\cT'$, and we refer to it as a {\it chain forest} for  $\cT$ and $\cT'$. Since the chain reduction can be performed in polynomial time~\cite{bordewich07b}, the chain forests $\cB_\cS$ and $\cB_\cT$ can also be calculated in polynomial time from $\cT$ and $\cT'$. Lastly, each element in $\cB_\cT$ whose members label the leaves of a common $n$-chain in $\cT$ and $\cT'$ with $n\geq 2$  is referred to as a {\it chain element}.

The next lemma bounds the number of elements in a chain forest.

\begin{lemma}\label{l:B_bound}
Let $\cT$ and $\cT'$ be two rooted binary phylogenetic $X$-trees. Let $\cS$ and $\cS'$ be the reduced tree pair with respect to $\cT$ and $\cT'$. Furthermore, let $\cB_\cS$ and $\cB_\cT$ be the chain forests for $\cS$ and $\cS'$, and $\cT$ and $\cT'$, respectively. Then $|\cB_\cT| =|\cB_\cS|< 5h(\cT,\cT')$.
\end{lemma}

\begin{proof}
By construction of $\cB_\cT$ from $\cB_\cS$, it immediately follows that $|\cB_\cT| =|\cB_\cS|$. To show that $|\cB_\cS|< 5h(\cT,\cT')$ let $\cH$ be a hybridization network that displays $\cS$ and $\cS'$ such that its number of hybridization vertices is minimized over all such networks. Furthermore, let $\cH'$ be the $h(\cH)$-generator associated with $\cH$.
As in the proof of Theorem \ref{thm:newKernel}, let $r_0$ be the number of node sides, i.e. reticulations with indegree~2 and outdegree~0, in $\cH'$ and let $r_1$ be the number of reticulations in $\cH'$ with indegree~2 and outdegree~1. Again, $r_0 + r_1 = h(\cH')= h(\cS,\cS')$. Recall that, to obtain~$\cH$ from~$\cH'$, we add one leaf to each node side of~$\cH'$, corresponding to a singleton in~$\cB_\cS$, and at most two leaves to each edge side of~$\cH'$. Each edge side of $\cH'$ to which we add two taxa corresponds to a 2-chain of $\cS$ and $\cS'$ and, therefore, to a single element in $\cB_\cS$. Hence, using \eqref{eq:edges} and Lemma~\ref{l:preserving}, we have
\[|\cB_\cT|=|\cB_\cS|  \leq  |E(\cH')| +r_0 = 5r_0 + 3r_1 -1 < 5(r_0+r_1) = 5 h(\cS,\cS') \leq 5 h(\cT,\cT').\]
\end{proof}

Consider again the chain forest $\cB_\cT$ for $\cT$ and $\cT'$. We define a {\it $\cB_\cT$-splitting} as an acyclic agreement forest for $\cT$ and $\cT'$ that can be obtained from $\cB_\cT$ by repeated replacements of a chain element $\{a_1,a_2,\ldots,a_n\}$ with the elements $\{a_1\},\{a_2\},\ldots,\{a_n\}$.

\begin{lemma}\label{l:bTprops}
Let $\cB_\cT$ be the chain forest for two rooted binary phylogenetic $X$-trees $\cT$ and $\cT'$. Let $\{a_1,a_2,\ldots,a_n\}$ be a chain  element in $\cB_\cT$, and let {$\cL_j$ be a non-chain element in $\cB_\cT$}. Furthermore, let $\cB_\cT'=(\cB_\cT-\{\{a_1,a_2,\ldots,a_n\}\})\cup\{\{a_1\},\{a_2\},\ldots,\{a_n\}\}$. Then
\begin{enumerate}
\item no directed cycle of $G_{\cB_\cT'}$ passes through an element of $\{\{a_1\},\{a_2\},\ldots,\{a_n\}\}$ and
\item no directed cycle of $G_{\cB_\cT}$ passes through {$\cL_j$}.
\end{enumerate}
\end{lemma}
\begin{proof}
{By the definition of $\cB_\cT$, note that $|\cL_j|=1$. If $\cL_j=\{\rho\}$, then the indegree of $\cL_j$ is 0 in $G_{\cB_\cT}$. Otherwise, if $\cL_j\ne\{\rho\}$, then its element labels a leaf of $\cT$ and $\cT'$  and, thus the outdegree of $\cL_j$ is 0 in $G_{\cB_\cT}$.} Furthermore, since each element in $\{\{a_1\},\{a_2\},\ldots,\{a_n\}\}$ also labels a leaf of $\cT$ and $\cT'$, the outdegree of the vertices $a_1,a_2,\ldots,a_n$ in $G_{\cB_\cT'}$ is 0. This establishes the lemma.
\end{proof}

Let $\opt$ denote the size of a $\cB_\cT$-splitting of smallest size.

\begin{lemma}\label{l:bSplitting}
Let $\cT$ and $\cT'$ be two rooted binary phylogenetic $X$-trees, and let $\cB_\cT$ be the chain forest for $\cT$ and $\cT'$. Then, $\opt<6h(\cT,\cT')$.
\end{lemma}
\begin{proof}
Let $\cF_{\cT}$ be a maximum acyclic agreement forest for $\cT$ and $\cT'$. In this proof, we see an agreement forest as a collection of trees (see the remark below the definition in Section~\ref{sec:prelim}). Thus, $\cF_{\cT}$ can be obtained from $\cT$ (or equivalently from $\cT'$) by deleting an $(|\cF_{\cT}|-1)$-sized subset, say $E_{\cF_\cT}$, of the edges of $\cT$ and cleaning up. Similarly, $\cB_\cT$ can be obtained from $\cT$ (or equivalently from $\cT'$) by deleting a $(|\cB_{\cT}|-1)$-sized subset, say $E_{\cB_\cT}$, and cleaning up. Now consider the forest $\cB_\cT'$ obtained from $\cT$ by removing the edge set $E_{\cF_{\cT}} \cup E_{\cB_{\cT}}$ and cleaning up.

We claim that $\cB_\cT'$ is a $\cB_\cT$-splitting. To see this, first observe that $\cB_\cT'$ is an acyclic agreement forest for $\cT$ and $\cT'$ because it can be obtained by removing edge set~$E_{\cB_{\cT}}$ from~$\cF_{\cT}$ and cleaning up. Hence, to show that $\cB_\cT'$ is a $\cB_\cT$-splitting, it is left to show that it can be obtained from $\cB_\cT$ by repeated replacements of a caterpillar on $\{a_1,a_2,\ldots,a_n\}$ by isolated vertices $\{a_1\},\{a_2\},\ldots,\{a_n\}$. By its definition, $\cB_\cT'$ can be obtained from $\cB_\cT$ by removing edges and cleaning up. Thus, what is left to prove is that each chain either survives or is atomized. For $n$-chains with~$n\geq 3$, this follows from Lemma~\ref{lem:survives}, and for $n=2$ it is clear because $\cB_\cT'$ can be obtained by removing edges from $\cB_\cT$ in which each 2-chain is a component on its own.

As the size of $\cB_\cT'$ is equal to the number of edges removed to obtain it from $\cT$ plus one, we have:
\[|\cB_\cT'| \leq |E_{\cF_{\cT}}| + |E_{\cB_{\cT}}| +1 = |\cF_{\cT}| - 1 + |\cB_{\cT}| < h(\cT,\cT') + 5 h(\cT,\cT') = 6 h(\cT,\cT'),\]
{where Lemma \ref{l:B_bound} is used to bound $|\cB_{\cT}|$.} This establishes the lemma.
\end{proof}

We are now in a position to prove the main result of this section.
\begin{proof}[Proof of Theorem~\ref{t:Hybrid2DFVS}]
Throughout this proof, let $n\geq 2$. Furthermore, let $\cB_\cT$ be the chain forest for $\cT$ and $\cT'$, and let $G$ be the graph obtained from the inheritance graph $G_{\cB_\cT}$ by subsequently
\begin{enumerate}
\item weighting each vertex that corresponds to a common $n$-chain $(a_1,a_2,\ldots,a_n)$ of $\cT$ and $\cT'$  with weight $n$;
\item deleting each vertex that  corresponds to a non-chain taxon in $\cB_\cT$; and
\item for each remaining vertex $v$, creating a new vertex $\bar{v}$ with weight $1$ and two new edges $(v,\bar{v})$ and $(\bar{v},v)$.
\end{enumerate}
Furthermore, let $w$ be the weight function on the vertices of $G$. See Figure~\ref{fig:todfvs1} for an example of the construction of~$G$. We call the added vertices $\bar{v}$ the \emph{barred} vertices of $G$. Note that each common $n$-chain of $\cT$ and $\cT'$ is represented by a vertex and its barred vertex in $G$. As $\cB_\cT$ can be calculated in polynomial time, the construction of $G$ also takes polynomial time, {and the size of $G$ is clearly polynomial in the cardinality of $\cB_\cT$}.

\begin{figure}
    \centering
    \includegraphics[scale=.5]{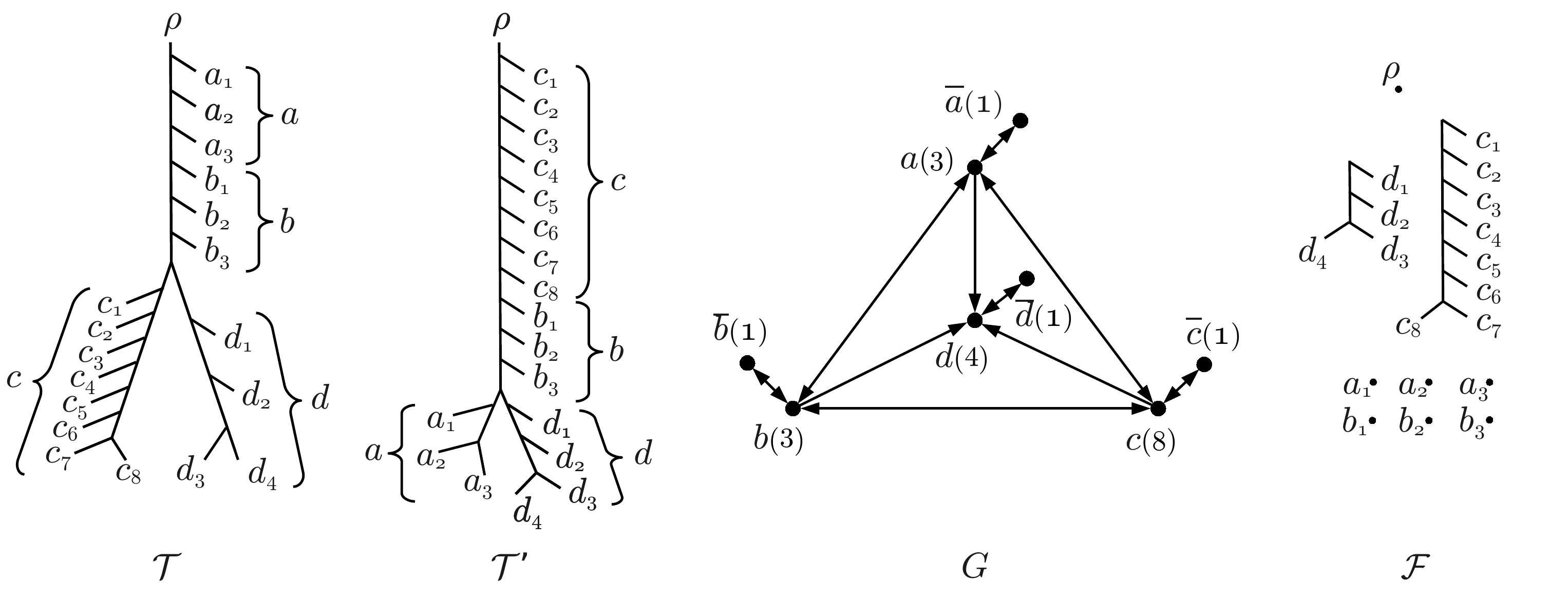}
    \caption{Two input trees~$\mathcal{T}$ and $\mathcal{T}'$, their auxiliary graph~$G$ (with weights between parentheses) and an acyclic agreement forest~$\mathcal{F}$ of~$\mathcal{T}$ and $\mathcal{T}'$. Note that~$\cF$ is a ${\cB_T\text{-splitting}}$ because it can be obtained from the chain forest $\cB_T$  by atomizing chains $a=(a_1,\ldots,a_3)$ and $b=(b_1,\ldots,b_3)$. Also note that~$\cF$ has~9 components, which is equal to the weight of a minimum feedback vertex set $\{a,b,\bar{c},\bar{d}\}$ of $G$, 8, plus a single non-chain taxon (in this case, $\rho$).
}
 \label{fig:todfvs1}
\end{figure}

Now, regarding $G$ as an instance of {\sc wDFVS}, we claim the following.

\noindent{\bf Claim.} There exists a $\cB_\cT$-splitting of size $k+s$, where $s$ is the number of non-chain elements in $\cB_\cT$, if and only if $G$ has a FVS of weight $k$.

{Suppose that $\cB_\cT'$ is a $\cB_\cT$-splitting of size~$k+s$. Hence, $k$ is equal to the number of chain elements in $\cB_\cT$ that are also elements in $\cB_\cT'$ plus the total number of leaves in common $n$-chains that are atomized in $\cB_\cT'$. Let $\bar{\cB}_\cT'$ be the forest that has been obtained from $\cB_\cT'$ by deleting all singletons, and let $G_{\bar{\cB}_\cT'}$ be its inheritance graph. Since $G_{\cB_\cT'}$ is acyclic, $G_{\bar{\cB}_\cT'}$ is also acyclic. Now, let $G'$ be the directed graph that has been obtained from $G$ in the following way. For each non-barred vertex $v$ in $G$, delete $v$ if $v$ corresponds to an $n$-chain of $\cT$ and $\cT'$ that is atomized in $\cB_\cT'$, and delete $\bar{v}$ if $v$ corresponds to an $n$-chain of $\cT$ and $\cT'$ that is not atomized in $\cB_\cT'$}. Note that for each 2-cycle $(v,\bar{v},v)$ of $G$ either $v$ or $\bar{v}$ is not a vertex of $G'$ because each $n$-chain that is common to $\cT$ and $\cT'$ is either atomized or not in $\cB_\cT'$. {This in turn implies that $G'$ is acyclic because  $G_{\bar{\cB}_\cT'}$ is isomorphic to $G'\backslash \bar{V}$, where $\bar{V}$ precisely contains all barred vertices of $G'$. Hence, an FVS of $G$, say $V$, contains each vertex of $G$ that is not a vertex of $G'$. Furthermore, by the weighting of $G$, it follows that the weight of $V$ is exactly $k$.}

Conversely, suppose that there exists an FVS of $G$, say $V$, with weight $k$. This implies that we can remove a set $V_1$ of barred vertices  and a set  $V_2=V\backslash V_1$ of non-barred vertices such that $\sum_{v_i\in V_2}w(v_i)+|V_1|=k$ and the graph  $G'=G\backslash V$ is acyclic. For each vertex $v_i\in V_2$, let $A_i=(a_{i,1},a_{i,2},\ldots,a_{i,n})$ be its associated common chain of $\cT$ and $\cT'$, and let $w(v_i)$ be the number of elements in $A_i$. Furthermore, let $V_1'$ be the subset of $V_1$ that contains precisely each vertex $\bar{v}$ of $V_1$ for which $v\notin V_2$. If $|V_1'|<|V_1|$, then it is easily checked that that $V_1'\cup V_2$ is an FVS of $G$ whose weight is strictly less than $k$. Therefore, we may assume for the remainder of this proof that $|V_1'|=|V_1|$. Now, let $\cB_\cT'$ be the forest that has been obtained from $\cB_\cT$ in the following way. For each vertex $v_i$ in $V_2$, replace $A_i$ in $\cB_\cT$ with the elements $\{a_{i,1}\},\{a_{i,2}\},\ldots,\{a_{i,n}\}$. Thus, $A_i$ is atomized in  $\cB_\cT'$. We next construct the inheritance graph $G_{\cB_{\cT}'}$ from $G_{\cB_{\cT}}$. For each vertex $v$ of $G_{\cB_{\cT}}$ that corresponds to a common $n$-chain $(a_1,a_2,\ldots,a_n)$ of $\cT$ and $\cT'$ that is atomized in $\cB_\cT'$, replace $v$ with the vertices $a_1,a_2,\ldots,a_n$, delete each edge $(v,w)$ of $G_{\cB_{\cT}}$, and replace each edge $(u,v)$ of $G_{\cB_{\cT}}$ with the edges $(u,a_1),(u,a_2),\ldots,(u,a_n)$. By Lemma~\ref{l:bTprops}, the vertices $a_1,a_2,\ldots,a_n$ have outdegree 0 in $G_{\cB_{\cT}'}$. Noting that there is a natural bijection between the cycles in $G_{\cB_{\cT}}$ and the cycles in $G$ that do not pass through any barred vertex, it follows that, as $G'$ is acyclic, $G_{\cB_{\cT}'}$ is also acyclic. Hence, $\cB_\cT'$ is a $\cB_\cT$-splitting for $\cT$ and $\cT'$. The claim now follows from
$$|\cB_\cT'|=s+\sum_{v_i\in V_2}w(v_i)+|V_1|=s+k.$$

It remains to show that the reduction is approximation preserving. Suppose that there exists a polynomial-time $c$-approximation for {\sc wDFVS}. Let $k$ be the weight of a solution returned by this algorithm, and let $k^*$ be the weight of an optimal solution. By the above claim, we can then construct a solution to {\sc MAAF} of size $k+s$, from which we can obtain a solution to {\sc MinimumHybridization} with value $k+s-1$ by Theorem~\ref{t:hybrid2}. We have,
$$k + s - 1 < ck^*+s\leq ck^*+cs=c(k^*+s)=c\cdot\opt$$ 
and, thus, a constant factor $c$-approximation for finding an optimal $\cB_T$-splitting. Now, by Lemma~\ref{l:bSplitting}, 
$$k + s - 1\leq c\cdot\opt\leq 6c\cdot h(\cT,\cT'),$$ 
thereby establishing that, if there exists a polynomial-time $c$-approximation for {\sc wDFVS}, then there exists a polynomial-time $6c$-approximation for {\sc MinimumHybridization}. This concludes the proof of the theorem.
\end{proof}

It is not too difficult to extend Theorem \ref{t:Hybrid2DFVS} to {\sc DFVS} i.e. the unweighted variant of directed feedback vertex set.

\begin{theorem}
\label{t:Hybrid2DFVSunweight} If, for some~$c\geq 1$, there exists a polynomial-time $c$-approximation for {\sc DFVS}, then there exists a polynomial-time $6c$-approximation for {\sc MinimumHybridization}.
\end{theorem}
\begin{proof}
In the proof of Theorem \ref{t:Hybrid2DFVS} we create an instance $G$ of {\sc wDFVS}. Let
$w$ be the weight function on the vertices of $G$. Note that the function is non-negative
and integral and for every vertex $v \in G$, $w(v) \leq |X|$ i.e. the weight function is polynomially bounded
in the input size. We create an instance $G'$ of DFVS as follows. For each vertex $v$ in $G$ we create $w(v)$ vertices in $G'$
$v_1, \ldots, v_{w(v)}$. For each edge $(u,v)$ in $G$ we introduce edges
$\{ (u_i, v_j) | 1 \leq i \leq w(u), 1 \leq j \leq w(v) \}$ in $G'$. Solutions to {\sc wDFVS($G$)}
and {\sc DFVS($G'$)} are very closely related, which allows us to use $G'$ and DFVS instead of
$G$ and wDFVS in the proof of Theorem \ref{t:Hybrid2DFVS}.\footnote{Formally, what
we demonstrate is an L-reduction from {\sc wDFVS} to {\sc DFVS} with coefficients
$\alpha = \beta = 1$ which works for instances with polynomially-bounded weights.} Specifically, consider any
feedback vertex set $F'$ of $G'$ of size $k$. We create a feedback vertex set $F$ of
$G$ as follows. For each vertex $v \in G$, we include $v$ in $F$ if and only if \emph{all}
the vertices $v_1, \ldots, v_{w(v)}$ are in $F'$. Note that the weight of $F$ is
less than or equal to $k$. To see that $F$ is a feedback vertex set, suppose
some cycle $C = u,v,w,\ldots,u$ survives in $G$. But then, for each vertex $u \in C$,
some vertex $u_i$ survives in $G'$, which means a cycle also survived in $G'$, contradicting
the assumption that $F'$ is a feedback vertex set. In the other direction, observe that any weight $k$ feedback vertex set $F$ of $G$ can be transformed into an feedback vertex set $F'$ of
$G'$ with size $k$ as follows: for each $v \in F$, place all $v_{1}, \ldots, v_{w(v)}$ in $F'$.
\end{proof}

Moreover, the reduction in the proof of Theorem~\ref{t:Hybrid2DFVS} can be used not only for constant~$c$, which we use in the next corollary.

\begin{corollary}
\label{cor:approxalg}
There exists a polynomial-time $\text{O}( \log r \log \log r)$-approximation for {\sc MinimumHybridization}, where
$r = h(\cT, \cT')$
\end{corollary}
\begin{proof}
In \cite{dfvsApprox}, which extended \cite{dfvsSeymour}, a polynomial-time approximation algorithm for wDFVS is presented whose approximation ratio is $\text{O}( \min( \log |V| \log \log |V|, \log \tau^{*} \log \log \tau^{*}) )$, where $|V|$ is the number of vertices in the wDFVS instance and $\tau^{*}$ is the optimal \emph{fractional} solution value of the problem. We show that in the wDFVS instance $G$ that we create in the proof of Theorem~\ref{t:Hybrid2DFVS},
both the number of vertices in $G$ and the weight of the optimal fractional solution value of {\sc wDFVS}$(G)$ are $\text{O}(r)$. To see that $G$ has at most $\text{O}(r)$ vertices, observe that $G$ contains two vertices for every chain element in
the chain forest $\cB_\cT$, and that (by Lemma \ref{l:B_bound}) $|\cB_\cT| < 5r$.
Secondly, recall from Lemma \ref{l:bSplitting} that $\opt < 6r$. By construction, $\opt$ is
an upper bound on the optimum solution value of {\sc wDFVS}$(G)$, hence on $\tau^{*}$. 
Thus, given $G$ as input, the algorithm in \cite{dfvsApprox} constructs
a feedback vertex set that is at most a factor $\text{O}( \log r \log \log r)$ larger than the true
optimal solution of {\sc wDFVS}$(G)$. As shown in the proof of Theorem
\ref{t:Hybrid2DFVS} this can be used to obtain an approximation ratio at most 6 times
larger for {\sc MAAF}, which is clearly also $\text{O}( \log r \log \log r)$.
\end{proof}

\noindent
Finally, note that for a given instance the actual approximation ratio obtained by Corollary \ref{cor:approxalg} will sometimes be determined by $|V|$, and sometimes by $\tau^{*}$,
and can potentially be significantly smaller than $\text{O}( \log r \log \log r)$. For example,
if there are very few chains in the chain forest, but they are all extremely long, then it can
happen that $|V| << \tau^{*}$. Conversely, if the chain forest contains many short chains, and
only a small number of them need to be atomized to attain acyclicity, then it can happen
that $\tau^{*} << |V|$.

\section{An approximation-preserving reduction from {\sc DFVS} to {\sc MinimumHybridization}}\label{sec:2minhybrid}

In this section we prove the following theorem.

\begin{theorem}
\label{theorem:dfvsToHybrid}
If, for some~$c\geq 1$, there exists a polynomial-time $c$-approximation algorithm for {\sc MinimumHybridization}, then there exists a  polynomial-time $(c+\epsilon)$-approximation algorithm for {\sc DFVS} for all~$\epsilon>0$.
\end{theorem}
\begin{proof}
We show an approximation preserving reduction from {\sc DFVS} to {\sc MAAF}. The theorem then follows because of the equivalence of MAAF and {\sc MinimumHybridization} described in Theorem~\ref{t:hybrid2}.

Let~$D=(V,A)$ be an instance of {\sc DFVS}. First we transform~$D$ into an auxiliary graph~$D'$. For a vertex~$v$ of~$D$, we denote the parents of~$v$ as $u_1,u_2,\ldots,u_{d^-(v)}$ and the children of~$v$ as $w_1,w_2,\ldots,w_{d^+(v)}$ (To facilitate the exposition, we assume a total order on the parents of each vertex and on the children of each vertex.). We construct the graph~$D'$ as follows. For every vertex~$v\in V$,~$D'$ has vertices $v_{\text{in}}^{u_1},v_{\text{in}}^{u_2},\ldots,v_{\text{in}}^{u_{d^-(v)}}$, vertices~$v^-$ and~$v^+$ as well as vertices $v_{\text{out}}^{w_1},v_{\text{out}}^{w_2},\ldots,v_{\text{out}}^{w_{d^+(v)}}$. The edges of~$D'$ are as follows. For each vertex~$v\in V$,~$D'$ has edges from each of $v_{\text{in}}^{u_1},v_{\text{in}}^{u_2},\ldots,v_{\text{in}}^{u_{d^-(v)}}$ to~$v^-$, an edge from~$v^-$ to~$v^+$ and edges from~$v^+$ to each of $v_{\text{out}}^{w_1},v_{\text{out}}^{w_2},\ldots,v_{\text{out}}^{w_{d^+(v)}}$. In addition, for each edge~$(u,v)$ of~$D$, there is an edge~$(u_\text{out}^{v},v_\text{in}^u)$ in~$D'$. This concludes the construction of~$D'$. An example is given in Figure~\ref{fig:reduction1}.

\begin{figure}
    \centering
    \includegraphics[scale=.5]{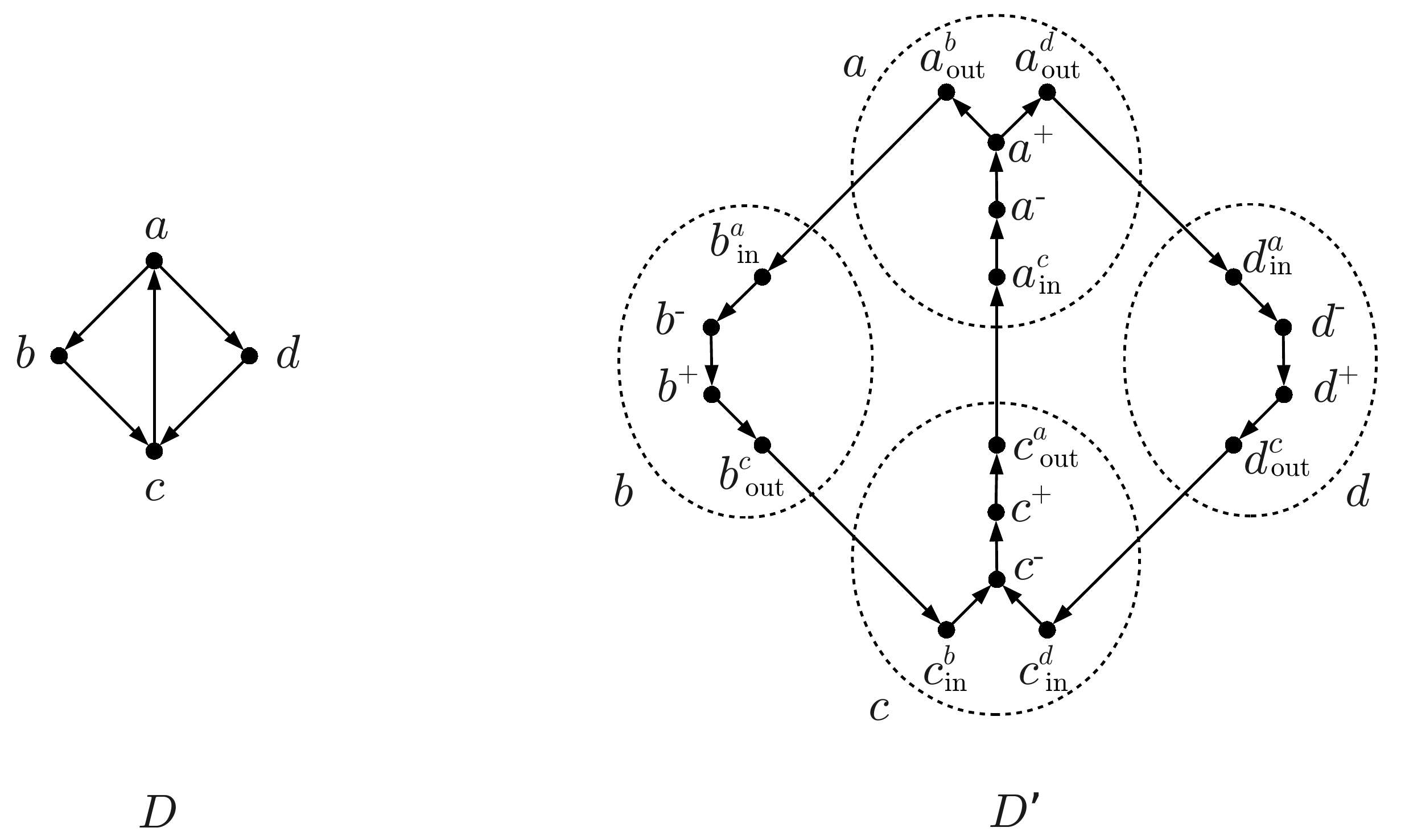}
    \caption{An instance~$D$ of DFVS and the modified graph~$D'$.}
    \label{fig:reduction1}
\end{figure}

We now first show that~$D$ has a FVS of size at most~$f$ if and only if~$D'$ has a FVS of size at most~$f$. Observe that each directed cycle of~$D$ corresponds to a directed cycle of~$D'$ and vice versa. Thus, from a FVS~$F$ of~$D$, we can construct a FVS~$F'$ of~$D'$ by, for each~$v\in F$, adding~$v^-$ to~$F'$. Reversely, from a FVS~$U'$ of~$D'$, we can create a FVS~$U$ of~$D$ as follows: a vertex~$v$ of~$D$ is put in~$U$ if and only if at least one of the corresponding vertices $v_{\text{in}}^{u_1},v_{\text{in}}^{u_2},\ldots,v_{\text{in}}^{u_{d^-(v)}}$,~$v^-$,~$v^+$,$v_{\text{out}}^{w_1},v_{\text{out}}^{w_2},\ldots,v_{\text{out}}^{w_{d^+(v)}}$ is in~$U'$.

Intuitively, the idea of our reduction is as follows. We will construct two rooted binary trees~$\cT$ and~$\cT'$ consisting of long chains. We build them in such a way that the graph~$D'$ is basically the inheritance graph of the chain forest for $\cT$ and $\cT'$. This graph can be made acyclic by atomizing some of the chains. Thus, solving {\sc DFVS} on~$D'$ is basically equivalent to deciding which chains to atomize. We make all the chains that can be atomized of the same length. Hence, since each chain that is atomized adds the same number of components to the agreement forest, solving {\sc DFVS} on~$D'$ is essentially equivalent to finding a maximum acyclic agreement forest for $\cT$ and $\cT'$.

Before we proceed, we need some more definitions. Recall that an {\it $n$-chain} of a tree is an $n$-tuple $(a_1,a_2,\ldots,a_n)$ of leaves such that the parent of $a_1$ is either the same as the parent of $a_2$ or the parent of $a_1$ is a child of the parent of $a_2$ and, for each $i\in\{2,3,\ldots,n-1\}$, the parent of $a_i$ is a child of the parent of $a_{i+1}$. A tree~$T$ whose leaf set~$\cL(T)$ is a chain of~$T$ is called a \emph{caterpillar} on~$\cL(T)$. It is easy to see that, for every chain~$C$, there exists a unique caterpillar on~$C$. By \emph{hanging} a chain~$C$ below a leaf~$x$, we mean the following: subdivide the edge entering~$x$ by a new vertex~$v$ and add an edge from~$v$ to the root of the caterpillar on~$C$. When we hang a chain~$C_1$ below a chain~$C_2$, we hang the caterpillar on~$C_1$ below the lowest leaf (or a lowest leaf)~$x_1$ of~$C_2$. By \emph{replacing} a leaf~$x$ by a chain~$C$ we mean: delete~$x$ and add an edge from its former parent to the root of the caterpillar on~$C$.

We are now ready to construct an instance of {\sc MAAF}. The trees,~$\cT$ and~$\cT'$, will be built of chains of three types: x-type, y-type and z-type. The x-type chains have length~$\ell$ while the y-type and z-type chains have length~$L$ (with $L>>\ell$). Each of these chains will be common to both trees. Recall that,
by Lemma~\ref{lem:survives}, we may assume that every chain either survives or is atomized. The idea is that y-type chains and z-type chains are so long that they will all survive. The x-type chains are shorter and might be atomized. In fact, the x-type chains that are atomized will correspond to a FVS of~$D'$.

\begin{figure}
    \centering
    \includegraphics[scale=.5]{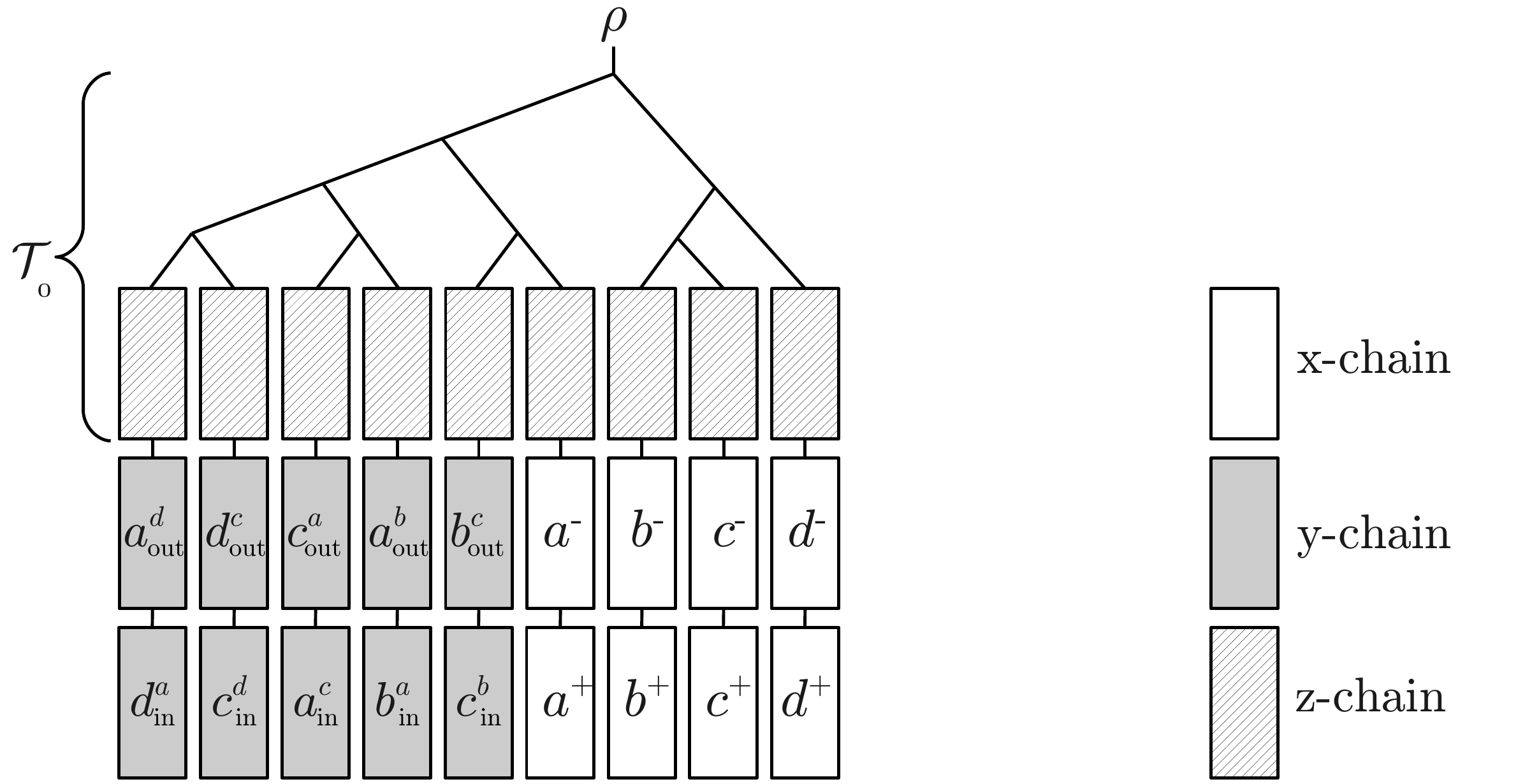}
    \caption{$\cT$: the first tree of the constructed MAAF instance.}
    \label{fig:reduction2}
\end{figure}

\begin{figure}
    \centering
    \includegraphics[scale=.5]{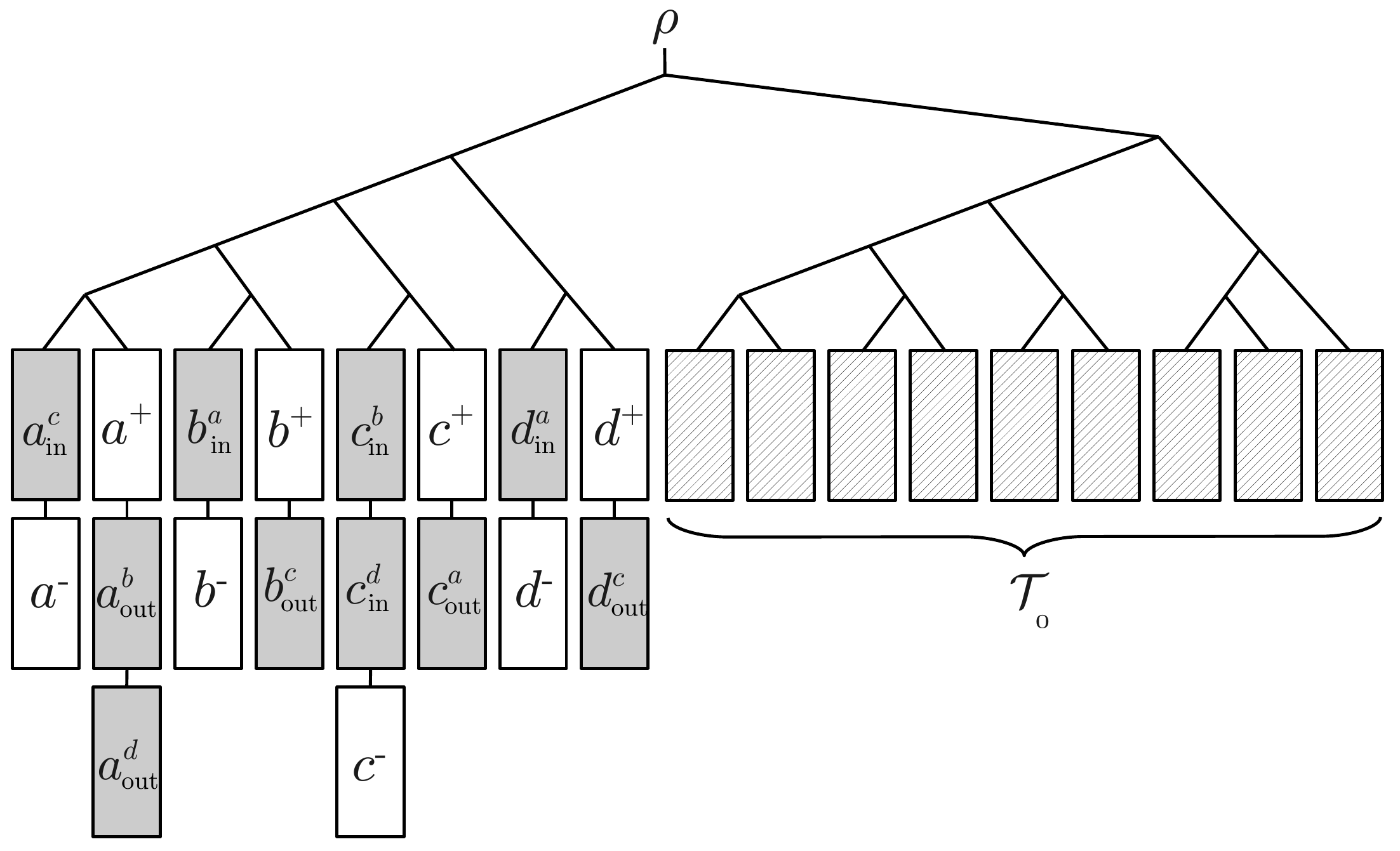}
    \caption{$\cT'$: the second tree of the constructed MAAF instance.}
    \label{fig:reduction3}
\end{figure}

We build the trees~$\cT$ and~$\cT'$ as follows. For each vertex of~$D'$ of the type~$v^-$ or~$v^+$ we create an x-type chain. For each other vertex of~$D'$ we create a y-type chain. Finally, for each vertex and edge of the original graph~$D$ we create a z-type chain. All leaves of all chains have different labels. Now we combine the chains into two trees as follows.

First~$\cT$. Start with an arbitrary rooted binary tree on~$|V|+|A|$ leaves and replace each leaf by a z-type chain. We call the current tree~$\mathcal{T}_0$. For each edge $(u,v)$ of~$D$, the tree contains a z-type chain. Hang below this z-type chain the y-type chain for~$u_\text{out}^v$ and below that the y-type chain for~$v_\text{in}^u$. Furthermore, for each vertex~$v$ of~$D$, the tree also has a z-type chain. Hang below this z-type chain the x-type chain for~$v^-$ and below that the x-type chain for~$v^+$.

Now~$\cT'$. Start with an arbitrary rooted binary tree on~$2|V|$ leaves. So we have two leaves for each vertex~$v$ of~$D$. Replace one of them by a concatenation of (from top to bottom) the y-type chains for $v_{\text{in}}^{u_1},v_{\text{in}}^{u_2},\ldots,v_{\text{in}}^{u_{d^-(v)}}$ and the x-type chain for~$v^-$. Replace the other leaf for~$v$ by a concatenation of (from top to bottom) the x-type chain for~$v^+$ and the y-type chains for $v_{\text{out}}^{w_1},v_{\text{out}}^{w_2},\ldots,v_{\text{out}}^{w_{d^+(v)}}$. Finally, hang a copy of~$\mathcal{T}_0$ below the root. This concludes the construction of the MAAF instance. For an example, see Figures~\ref{fig:reduction2} and~\ref{fig:reduction3}.

We claim that~$D'$ (and thus~$D$) has a FVS of size at most~$f$ if and only if there exists an acyclic agreement forest of~$\cT$ and~$\cT'$ of size at most $1+2(|A|+|V|)+(\ell-1)f$.

To show this, consider the agreement forest~$A_D$ for~$\cT$ and~$\cT'$ in which~$\mathcal{T}_0$ is one component, each x-type chain is one component, and each y-type chain is one component. The inheritance graph $G_{A_D}$ of this agreement forest can be obtained by making some (minor) changes to~$D'$. Add a vertex labelled~$\mathcal{T}_0$ with edges to all other vertices. Secondly, for each~$v\in V$, add an edge $(v_{\text{in}}^{u_i},v_{\text{in}}^{u_j})$ for each pair~$i,j$ with $1\leq i<j\leq d^-(v)$ and an edge $(v_{\text{out}}^{w_i},v_{\text{out}}^{w_j})$ for each pair~$i,j$ with $1\leq i<j\leq d^+(v)$. Observe that, given a FVS of~$D'$, there exists a FVS of~$D'$ of at most the same size that consists of only vertices of the type~$v^-$. Such a FVS is also a FVS of $G_{A_D}$ since any directed cycle passing through any of the newly added edges $(v_{\text{in}}^{u_i},v_{\text{in}}^{u_j})$ or $(v_{\text{out}}^{w_i},v_{\text{out}}^{w_j})$ also passes through $v^-$. Thus, if we consider (without loss of generality) only FVSs consisting of $v^-$-type vertices, then any FVS of~$D'$ is a FVS of~$G_{A_D}$ and vice versa. In addition, since $v^-$-type vertices correspond to x-type chains, it is possible to make~$G_{A_D}$ acyclic by atomizing only x-type chains.

Let~$F$ be a FVS of~$D$ and let~$F'$ (as before) be the corresponding FVS of~$D'$ that contains only vertices of the type~$v^-$. Then we can construct an agreement forest~$\cR$ of~$\cT$ and~$\cT'$ as follows. One component consists of the tree~$\mathcal{T}_0$. Each of the y-type chains is also one component, as well as the x-type chains that do not correspond to vertices in~$F' $. Finally, for each other x-type chain (that \emph{does} correspond to a vertex in~$F'$), we create a separate component for each leaf. Thus, the number of components is $1+2|A|+(2|V|-|F'|)+\ell|F'| = 1+2(|A|+|V|)+(\ell-1)|F|$. We have to show that the inheritance graph $G_\cR$ is acyclic. We can construct $G_\cR$ from~$G_{A_D}$ as follows. Delete every vertex~$v^-\in F'$ and instead add a vertex for each leaf of the corresponding x-type chain with incoming edges from~$\mathcal{T}_0$ and from $v_{\text{in}}^{u_1},v_{\text{in}}^{u_2},\ldots,v_{\text{in}}^{u_{d^-(v)}}$. Since we only introduced leaves with incoming edges, this modification does not create any directed cycles. Thus, since~$F'$ contains a vertex of each directed cycle of~$G_{A_D}$, and all vertices from~$F'$ have been removed, $G_\cR$ is acyclic. It follows that~$\cR$ is an acyclic agreement forest for~$\cT$ and~$\cT'$.

To show the other direction, let~$\cA$ be an acyclic agreement forest of~$\cT$ and~$\cT'$. We may assume that all y-type chains and z-type chains survive in~$\cA$, since we can choose~$L$ sufficiently large. To see this, recall that we may assume by Lemma~\ref{lem:survives} that each chain either survives or is atomized. Hence, if a y-type chain or z-type chain does not survive, it is atomized and adds~$L$ components to the agreement forest. Thus, by choosing~$L$ large enough (as will be specified later) we can make sure that all y-type chains and z-type chains survive. Secondly, observe that we may in addition assume that all z-type chains are together in a single component (if they are not, we can put them together and reduce the number of components). Now consider two chains that are not both z-type chains. We show that these chains can not be together in a single component of~$\cA$. Firstly, if the two chains are below each other in~$\cT$, then they are next to each other in~$\cT'$. Secondly, if the two chains are next to each other in~$\cT$, then they are separated by a z-type chain in~$\cT$ but not in~$\cT'$. Hence, by~(2) in the definition of an agreement forest, the two chains can not be together in a single component of~$\cA$. Thus, the components of~$\cA$ are as follows. Tree~$\mathcal{T}_0$ is the component containing the root and all z-type chains. Furthermore, each y-type chain, each surviving x-type chain, and each leaf of a non-surviving x-type chain is a separate component. Let~$\tilde{F}$ be the set of vertices of~$G_{A_D}$ corresponding to the non-surviving x-type chains. Thus, each vertex in~$\tilde{F}$ is of the type~$v^-$ or~$v^+$. We will show that~$\tilde{F}$ is a FVS of~$G_{A_D}$ and hence of~$D'$. We can construct $G_{\cA}$ from $G_{A_D}$ as follows. Remove each vertex in~$\tilde{F}$ from~$G_{A_D}$ and add each leaf of the corresponding x-type chain as a separate vertex. Then add edges to these newly added vertices (these edges are not important since they do not create any directed cycles). Since~$\cA$ is an acyclic agreement forest, $G_\cA$ is acyclic and hence~$\tilde{F}$ is a FVS. The size~$|\tilde{F}|$ of the FVS is equal to the number of non-surviving x-type chains. Thus, $|\cA| = 1+2|A|+(2|V|-|\tilde{F}|)+\ell|\tilde{F}| = 1+2(|A|+|V|)+(\ell-1)|\tilde{F}|$.

The reduction is clearly polynomial time. It remains to show that it is approximation preserving. Suppose that there exists a $c$-approximation algorithm for {\sc MAAF}. Say that~$m$ is the size of the MAAF returned by this algorithm and~$m^*$ the size of an optimal solution. Recall that {\sc MAAF} minimizes the size of an agreement forest minus one, so $m-1\leq c\cdot (m^*-1)$. We have shown that~$D$ has a FVS of size at most~$f$ if and only if~$\cT$ and~$\cT'$ have an acylic agreement forest of size at most $1+2(|A|+|V|)+(\ell-1)f$. Thus, $m^*=1+2(|A|+|V|)+(\ell-1)f^*$. Moreover, an approximate solution~$f$ of DFVS can be computed from an approximate solution~$m$ of MAAF by taking $f = ( m - 1 - 2(|A|+|V|) ) / (\ell-1)$. Then we have

\begin{eqnarray*}
f & = & \frac{m - 1 - 2(|A|+|V|)}{\ell-1}\\
& \leq & \frac{c\cdot (m^* - 1) - 2(|A|+|V|)}{\ell-1}\\
& = & \frac{c(2(|A|+|V|)+(\ell-1)f^*)-2(|A|+|V|)}{\ell-1}\\
& = & c\cdot f^* + \frac{2(c-1)(|A|+|V|)}{\ell-1}\\
& = & c\cdot f^* + 1\\
\end{eqnarray*}

if we take $\ell = 2(c-1)(|A|+|V|)+1$. We still need to specify the value of~$L$, which needs to be sufficiently large so that all y-type chains and z-type chains survive. Since any graph trivially has a FVS of size~$|V|$, any constructed {\sc MAAF} instance has $m^*\leq 1+2(|A|+|V|)+(\ell-1)|V|$. Thus, a $c$-approximation algorithm will return an acyclic agreement forest of size~$m$ with $m-1\leq c(m^*-1) \leq c(2(|A|+|V|)+(\ell-1)|V|)$. And hence with~$m\leq c(2(|A|+|V|)+(\ell-1)|V|)+1$. So it suffices to take $L = c(2(|A|+|V|)+(\ell-1)|V|) + 2 = 2c(|A|+|V|)(1 + (c-1)|V|) + 2$.

Now take $\epsilon > 0$. If $f^*<1 / \epsilon$, we can compute an optimal solution for {\sc DFVS} by brute force in polynomial time. Otherwise, $1 \leq  \epsilon\cdot f^*$ and we have

\[
f \leq c\cdot f^* + \epsilon\cdot f^* = (c+\epsilon)f^*.
\]

Thus, if there exists a $c$-approximation for MAAF, then there exists a $(c+\epsilon)$-approximation for DFVS for every fixed~$\epsilon > 0$.

\end{proof}

In contrast to the result in Section~\ref{sec:2dfvs}, the reduction above can only be used for constant~$c$. It does \emph{not} show that e.g. an $\text{O}(\log |X|)$-approximation for {\sc MinimumHybridization} would imply an $\text{O}(\log |V|)$-approximation for {\sc DFVS}. Hence, it is indeed possible that {\sc MinimumHybridization} admits an $\text{O}(\log |X|)$-approximation while {\sc DFVS} does not admit an $\text{O}(\log |V|)$-approximation. For neither of the problems such an approximation is known to exist.

Finally, we note that Theorem \ref{theorem:dfvsToHybrid} also allows
us to improve upon the best-known inapproximability result for {\sc MinimumHybridization}.

\begin{corollary}
There does not exist a polynomial-time $c$-approximation for {\sc MinimumHybridization}, where
$c < 10\sqrt{5}-21 \approx 1.3606$, unless P=NP. If the Unique Games Conjecture
holds, then there does not exist a polynomial-time $c$-approximation for {\sc MinimumHybridization} where
$c < 2$.
\end{corollary}
\begin{proof}
In \cite{karp1972} a simple reduction is shown from the problem \textsc{Vertex Cover}
to the problem {\sc DFVS}. Specifically, given an undirected graph $G$ as input to \textsc{Vertex Cover} we create a directed graph $G'$ by transforming each edge $\{u,v\}$ in $G$ into two directed edges
$(u,v), (v,u)$ in $G'$. It is easy to show that $G'$ has a feedback vertex set of size
$k$ if and only if $G$ has a vertex cover of size $k$. Consequently, any polynomial-time
$c$-approximation algorithm for DFVS can be used to construct a polynomial-time
$c$-approximation for \textsc{Vertex Cover}. The latter problem does not permit a polynomial-time $c$-approximation, for any $c < 10\sqrt{5}-21 \approx 1.3606$, unless P=NP \cite{dinur,DinurAnnals}.
Also, it has been shown that if the Unique Games Conjecture is true then no approximation better than 2 is possible \cite{Khot2008}. Now, the proof of Theorem \ref{theorem:dfvsToHybrid} shows that, if there exists a $c$-approximation for {\sc MinimumHybridization}, then there exists a $(c+\epsilon)$-approximation for DFVS for \emph{every} fixed~$\epsilon > 0$. Hence the existence of a $c$-approximation for {\sc MinimumHybridization} where $c < 10\sqrt{5}-21$ (respectively, $c < 2$) would mean the existence of a $c'$-approximation for DFVS (and thus also for \textsc{Vertex Cover}) where $c' <  10\sqrt{5}-21$ (respectively, $c' < 2$).
\end{proof}

\bibliographystyle{plain}
\bibliography{CycleKiller_v8}

\end{document}